\documentclass[letterpaper, 10 pt, conference]{ieeeconf}

\usepackage[pdftex]{graphicx}
\usepackage{epstopdf}
\usepackage{subfig}
\usepackage{amsmath}
\usepackage{mathrsfs}
\usepackage{color}
\usepackage{balance}
\usepackage{amssymb}
\newtheorem{theorem}{Theorem}
\newtheorem{lemma}{Lemma}

\IEEEoverridecommandlockouts
\title{\LARGE \bf
Output-Feedback Stabilization for a Class of Linear Parabolic Systems
}

\author{Agus Hasan
\thanks{The author is with Maersk Mc-Kinney Moller Institute,
        University of Southern Denmark, 5230 Odense, Denmark.
        Email: {\tt\small agha@mmmi.sdu.dk}.}
}

\begin{document}

\maketitle
\thispagestyle{empty}
\pagestyle{empty}

\begin{abstract}
We consider output-feedback stabilization problems for a class of two-component linear parabolic systems with boundary actuation and measurement. The state-feedback control laws are obtained using backstepping method and require measurement of the state at each point in the domain. To this end, backstepping observers are designed for both anti-collocated and collocated sensors and actuators. Furthermore, we show the closed-loop systems consisting of the plant, the backstepping control laws, and the observer is exponentially stable. The backstepping method is used to obtain both control and observer kernels. The kernels are the solution of $4\times4$ systems of second-order hyperbolic linear PDEs whose well-posedness is shown.
\end{abstract}

\section{INTRODUCTION}

Many physical processes can be modeled by parabolic differential equations such as heat diffusion, flow in porous media, and ocean acoustic propagation. Some control methods have also been developed and depending on where the actuators are placed, the control methods can be distinguished as boundary control and in-domain control. A successful control design for in-domain control has been developed by \cite{Cris}. In many applications, however, boundary control is generally considered to be physically more realistic since actuation and sensing are usually located at the boundaries. It is also a harder problem compared to an in-domain control since the input and output operators are unbounded operators.

In control literatures, the boundary control problems for parabolic PDEs are subject to different approaches, e.g., finite-element approximation was proposed by \cite{Lasiecka}, a semigroup approach was proposed by \cite{Tri}, while \cite{Seid} studied null-controllability of the boundary control problem. A more recent result on tracking control for boundary controlled of parabolic PDEs can be found in \cite{Meurer}.

A systematic method for boundary control of PDEs was developed by \cite{Krstic}. The method is called backstepping and is primarily used for nonlinear ODE systems in strict-feedback form. Thus, the method represents a major shift from finite-dimensional to infinite-dimensional systems. The backstepping method has been successfully used as a tool for control design and state estimation of many type of linear PDEs such as transport equation (\cite{Mail1,Mail2,Mail3}), wave equation (\cite{Mail4,Mail5}), Kortweg-de Vries equation (\cite{Mail6,Mail7}), and a more relevant to the present paper is the reaction-diffusion equation (parabolic type of PDEs). The method is also has been extended to nonlinear parabolic systems with Volterra nonlinearities \cite{Rafa}.

The method uses change of variable by shifting the system state using a Volterra operator. A property of Volterra operator is that the state transformation is triangular which ensures the invertibility of the change of the variable. Furthermore, using method of successive approximation \cite{Andrey2} or Marcum Q-functions \cite{Vazquez2} one may found an explicit expression for the transformation kernel. Thus, the feedback law can be constructed explicitly and the closed-loop solutions can be found in closed-form.

The paper is organized as follows. The problem is stated in section II. In section III, using the backstepping method, we derived a state-feedback control law and proving exponential stability of the closed-loop system. In section IV, we present our boundary observer and prove convergence of observer estimates to the real states. Furthermore, two cases are considered. The first case is when the sensor and the actuator are placed at different boundary (anti-collocated setup) and the second case is when both of them are placed at the same boundary (collocated setup). In section V we combine the full-state feedback control law and the observer to obtain an output-feedback control law. The well-posedness of the kernel equations is presented in section VI. Finally, conclusions are stated in section VII.

\section{PROBLEM FORMULATION}

We consider output-feedback stabilization problems for a class of two-component linear parabolic systems
\begin{eqnarray}
w_t(x,t)&=&w_{xx}(x,t)+\Sigma w(x,t)\label{main1}
\end{eqnarray}
where
$
w(x,t)= \begin{pmatrix}
  u(x,t)\;\;v(x,t)
 \end{pmatrix}^T$, $x\in[0,1]$, $t>0$, and $\Sigma= \begin{pmatrix}
  0 & \lambda_1 \\
  \lambda_2 & 0
 \end{pmatrix}
$. Systems (\ref{main1}) arise naturally in systems consisting of two interacting components and widely used to describe phenomena in variety of biological, chemical, and physical systems. For instance, (\ref{main1}) can be used to model Turing instability \cite{Turing}, instability of slime mold amoebae aggregation \cite{Keller}, and chemical reaction of two components \cite{Prigo}. The boundary conditions are given as\footnote{The notations $\mathbf{0}$ and $\mathbf{1}$ will be used to denote vectors with all entries 0 and 1, respectively.}
\begin{eqnarray}
w_x(0,t)=\mathbf{0}\;\;\;w(1,t)=U(t)=\begin{pmatrix}
  U_1(t) \\
  U_2(t)
 \end{pmatrix}\label{bc2}
\end{eqnarray}
where $U_1$ and $U_2$ are control laws which should be designed such that the equilibrium of the system (\ref{main1}) is exponentially stable in some norms (e.g., the $L^2$-norm).

\section{FULL-STATE FEEDBACK CONTROL LAW}

The main idea of the backstepping method is to design control law by mapping the original system (\ref{main1}) into the so-called target system. The coordinate transformation is invertible so that the stability of the target system translates into stability of the closed-loop system consisting of the plant plus boundary feedback \cite{Krstic}. In this case, we choose the following target system

\begin{eqnarray}
\gamma_t(x,t)&=&\gamma_{xx}(x,t)\label{maint1}\\
\gamma_x(0,t)&=&\mathbf{0}\;\;\;\gamma(1,t)=\mathbf{0}\label{tbc1}
\end{eqnarray}

where $
\gamma(x,t)= \begin{pmatrix}
  \alpha(x,t) \\
  \beta(x,t)
\end{pmatrix}$. Using Lyapunov stability analysis, it is easy to show that for any initial condition $(\alpha_0,\beta_0)\in L^2(0,1)$, the equilibrium of the target system (\ref{maint1}) with boundary conditions (\ref{tbc1}) is exponentially stable in the sense of $L^2$-norm. Note that using separation of variables it is also easy to obtain the explicit solution of the target system.

\subsection{Backstepping transformation}

To map the original system (\ref{main1}) into the target system (\ref{maint1}), we use the following integral operator state transformations \cite{Vazquez}
\begin{eqnarray}
\gamma(x,t)=w(x,t)-\int_0^x\! K(x,y)w(y,t) \,\mathrm{d}y\label{trans1}
\end{eqnarray}
where $K(x,y)= \begin{pmatrix}
  K^{uu}(x,y) & K^{uv}(x,y) \\
  K^{vu}(x,y) & K^{vv}(x,y)
 \end{pmatrix}$. Substituting transformation (\ref{trans1}) into target system (\ref{maint1})-(\ref{tbc1}), using plant equation (\ref{main1})-(\ref{bc2}), and integration by parts, one obtains the equations that the kernels must satisfy
\begin{eqnarray}
\bar{K}_{xx}(x,y)-\bar{K}_{yy}(x,y)&=&\Lambda\bar{K}(x,y)\label{ker1}\\
\bar{K}(x,x)&=&\left(0\;-\frac{\lambda_1}{2}x\;-\frac{\lambda_2}{2}x\;0\right)^T\label{gau}\\
\bar{K}_y(x,0)&=&\mathbf{0}\label{bcker6}
\end{eqnarray}
where
\begin{eqnarray}
\bar{K}(x,y)&=&\left(K^{uu}(x,y)\;K^{uv}(x,y)\;K^{vu}(x,y)\;K^{vv}(x,y)\right)^T\\
\Lambda&=&\begin{pmatrix}
0&\lambda_2&0&0\\
\lambda_1&0&0&0\\
0&0&0&\lambda_2\\
0&0&\lambda_1&0
\end{pmatrix}
\end{eqnarray}
Note that, this system evolving on triangular domain $\Upsilon=\{(x,y)\in\mathbb{R}^2|0\leq y\leq x\leq1\}$ since boundary condition (\ref{gau}) is on the characteristic (Goursat-type).

\subsection{Inverse transformation}

To ensure the target system and the closed-loop system have equivalent stability properties, integral transformation (\ref{trans1}) has to be invertible. In our case, the inverse is given by
\begin{eqnarray}
w(x,t)=\gamma(x,t)+\int_0^x\! L(x,y)\gamma(y,t) \,\mathrm{d}y\label{invtrans1}
\end{eqnarray}
where $L(x,y)= \begin{pmatrix}
  L^{\alpha\alpha}(x,y) & L^{\alpha\beta}(x,y) \\
  L^{\beta\alpha}(x,y) & L^{\beta\beta}(x,y)
\end{pmatrix}$. Again, substituting transformation (\ref{invtrans1}) into target system (\ref{maint1})-(\ref{tbc1}), using plant equation (\ref{main1})-(\ref{bc2}), and integration by parts, one obtains the equations that the inverse kernels must satisfy
\begin{eqnarray}
\bar{L}_{xx}(x,y)-\bar{L}_{yy}(x,y)&=&-\Psi\bar{L}(x,y)\\
\bar{L}(x,x)&=&\left(0\;-\frac{\lambda_1}{2}x\;-\frac{\lambda_2}{2}x\;0\right)^T\\
\bar{L}_y(x,0)&=&\mathbf{0}
\end{eqnarray}
where
\begin{eqnarray}
\bar{L}(x,y)&=&\left(L^{\alpha\alpha}(x,y)\;L^{\alpha\beta}(x,y)\;L^{\beta\alpha}(x,y)\;L^{\beta\beta}(x,y)\right)^T\\
\Psi&=&\begin{pmatrix}
0&0&\lambda_1&0\\
0&0&0&\lambda_1\\
\lambda_2&0&0&0\\
0&\lambda_2&0&0
\end{pmatrix}
\end{eqnarray}
Existence of $\bar{K}(x,y)$ and $\bar{L}(x,y)$ are given in section 6.

\subsection{Control law}

If the integral transformation (\ref{trans1}) evaluated at $x=1$, we get
\begin{eqnarray}
U(t)=\int_0^1\! K(1,y)w(y,t) \,\mathrm{d}y\label{con1}
\end{eqnarray}
This control law is our full-state feedback control law which requires the measurement of the state at each point in the domain. The summary of this section is stated in the following theorem.
\begin{theorem}
Let $K(x,y)$ be the solution of (\ref{ker1})-(\ref{bcker6}). Then for any initial condition $w_0\in L^2(0,1)$ which compatible with boundary condition (\ref{bc2}), system (\ref{main1}) has a unique classical solution $w\in C^{2,1}((0,1)\times(0,\infty))$. Additionally, the origin, $w\equiv0$, is exponentially stable in the $L^2$-sense.
\end{theorem}
\begin{proof}
The well-posedness problem can be solved by showing the transformation kernel $K(x,y)$ and the inverse kernel $L(x,y)$ are exists. This is discussed in section 6 where we show there exists a continuous twice-differentiable function $K(x,y)$ that satisfies (\ref{ker1})-(\ref{bcker6}). Remark that since the target system is exponentially stable in the $L^2$-norm, then there exists $\kappa$ and $\epsilon$ such that
\begin{eqnarray}
\|\gamma(\cdot,t)\|_{L^2}&\leq&\kappa\|\gamma(\cdot,0)\|_{L^2}e^{-\epsilon t}
\end{eqnarray}
Since the kernels are continuous, define the following upper bounds $K_{\infty},L_{\infty}>0$
\begin{eqnarray}
K_{\infty}^2&=&\max_{x,y\in[0,1]}\|K(x,y)\|_2^2\\
L_{\infty}^2&=&\max_{x,y\in[0,1]}\|L(x,y)\|_2^2
\end{eqnarray}
From (\ref{invtrans1}), we have
\begin{eqnarray}
&&\|w(\cdot,t)\|_{L^2} \nonumber\\
&\leq& \|\gamma(\cdot,t)\|_{L^2}+\left\|\int_0^x\! L(\cdot,y)\gamma(y,t) \,\mathrm{d}y\right\|_{L^2}\nonumber\\
&\leq& \|\gamma(\cdot,t)\|_{L^2}+\sqrt{\int_0^1\!\left\|\int_0^x\! L(x,y)\gamma(y,t) \,\mathrm{d}y\right\|_2^2\mathrm{d}x}\nonumber\\
&\leq& \|\gamma(\cdot,t)\|_{L^2}+\sqrt{\int_0^1\!\int_0^x\! \left\|L(x,y)\gamma(y,t)\right\|_2^2 \,\mathrm{d}y\mathrm{d}x}\nonumber\\
&\leq& \left(1+L_{\infty}\right)\|\gamma(\cdot,t)\|_{L^2}\nonumber\\
&\leq& \kappa\left(1+L_{\infty}\right)\left(1+K_{\infty}\right)\|w(\cdot,0)\|_{L^2}e^{-\epsilon t}
\end{eqnarray}
which concludes the proof.
\end{proof}

\section{BOUNDARY OBSERVER}

Since the state-feedback control law (\ref{con1}) requires measurements of the state at each point in the domain, we need to design a state observer. We consider two setups: the anti-collocated setup (when sensor and actuator are placed at the opposite ends) and the collocated setup (when sensor and actuator are placed at the same ends).

\subsection{Anti-Collocated Sensor and Actuator}

Let's assume that only $u(0,t)$ is available for measurement. Using this boundary information, it is possible to reconstruct the state in the domain. Following \cite{Andrey1}, the observer is designed for the plant is as follow
\begin{eqnarray}
\hat{w}_t(x,t)&=&\hat{w}_{xx}(x,t)+\Sigma\hat{w}(x,t)+p(x)\tilde{u}(0,t)\label{obsmain1}\\
\hat{w}_x(0,t)&=&\begin{pmatrix}
  L\tilde{u}(0,t) \\
  0
 \end{pmatrix}\;\;\;\hat{w}(1,t)=U(t)\label{obsbc4}
\end{eqnarray}
where $\tilde{u}=u-\hat{u}$. Here, the functions $p(x)=\begin{pmatrix}
  p_1(x) \\
  p_2(x)
 \end{pmatrix}$, and the constant $L$ are observer gains to be determined. Remark that the state observer is just the plant plus output injection. The objective now is to find those gains such that $\hat{u}$ converges to $u$ as time goes to infinity. To this end, consider the following error system
\begin{eqnarray}
\tilde{w}_t(x,t)&=&\tilde{w}_{xx}(x,t)+\Sigma\tilde{w}(x,t)-p(x)\tilde{u}(0,t)\label{anerror1}\\
\tilde{w}_x(0,t)&=&\begin{pmatrix}
  -L\tilde{u}(0,t) \\
  0
 \end{pmatrix}\;\;\;\hat{w}(1,t)=\mathbf{0}\label{anerrorbc4}
\end{eqnarray}
Using the following (invertible) transformation
\begin{eqnarray}
\tilde{w}(x,t)=\tilde{\gamma}(x,t)-\int_0^x\! P(x,y)\tilde{\gamma}(y,t)\,\mathrm{d}y\label{acoltrans1}
\end{eqnarray}
where $P(x,y)= \begin{pmatrix}
  P^{uu}(x,y) & P^{uv}(x,y) \\
  P^{vu}(x,y) & P^{vv}(x,y)
 \end{pmatrix}$, we transform the error system into the following exponentially target system
\begin{eqnarray}
\tilde{\gamma}_t(x,t)&=&\tilde{\gamma}_{xx}(x,t)\label{targ1}\\
\tilde{\gamma}_x(0,t)&=&\mathbf{0}\;\;\;\tilde{\gamma}(1,t)=\mathbf{0}\label{bctarg1}
\end{eqnarray}
 By plugging the transformation (\ref{acoltrans1}) into the error system (\ref{anerror1})-(\ref{anerrorbc4}), using the target system (\ref{targ1})-(\ref{bctarg1}), and integration by parts, the kernels must satisfy the following system
\begin{eqnarray}
\bar{P}_{xx}(x,y)-\bar{P}_{yy}(x,y)&=&-\Psi\bar{P}(x,y)\label{anobsker1}\\
\bar{P}(x,x)&=&\left(0\;\frac{\lambda_1}{2}(1-x)\;\frac{\lambda_2}{2}(1-x)\;0\right)^T\\
\bar{P}(1,y)&=&\mathbf{0}\label{anobsbc4}
\end{eqnarray}
where
\begin{eqnarray}
\bar{P}(x,y)=\left(P^{uu}(x,y)\;P^{uv}(x,y)\;P^{vu}(x,y)\;P^{vv}(x,y)\right)^T
\end{eqnarray}
The gains are given by
\begin{eqnarray}
p_1(x)&=&P_y^{uu}(x,0),\;\;p_2(x)=P_y^{vu}(x,0),\;\;L=0\label{angain}
\end{eqnarray}
Hence, we can state the following theorem.
\begin{theorem}
Let $P(x,y)$ be the solution to the system (\ref{anobsker1})-(\ref{anobsbc4}). Then for any initial condition $\tilde{w}_0\in L^2(0,1)$ which compatible with the boundary condition (\ref{anerrorbc4}), the system (\ref{anerror1}) with $p_1(x)$, $p_2(x)$, and $L$ are given by (\ref{angain}), has a unique classical solution $\tilde{w}\in C^{2,1}((0,1)\times(0,\infty))$. Additionally, the origin, $\tilde{w}\equiv0$, is exponentially stable in the $L^2$-sense.
\end{theorem}

\subsection{Collocated Sensor and Actuator}

If the only available measurement is at $x=1$, the observer design is slightly different. Let us assume $u_x(1,t)$ as the measurement. The state observer is given by
\begin{eqnarray}
\hat{w}_t(x,t)&=&\hat{w}_{xx}(x,t)+\Sigma\hat{w}(x,t)+p(x)\tilde{u}_x(1,t)\label{colerror1}\\
\hat{w}_x(0,t)&=&\mathbf{0}\;\;\;\hat{w}(1,t)=U(t)+\begin{pmatrix}
L\tilde{u}_x(1,t)\\
0
\end{pmatrix}\label{colerrorbc4}
\end{eqnarray}
Remark that the output injection is placed at the same boundary at which the actuator is located. Let us consider the following error system
\begin{eqnarray}
\tilde{w}_t(x,t)&=&\tilde{w}_{xx}(x,t)+\Sigma\tilde{w}(x,t)-p(x)\tilde{u}_x(1,t)\label{hehe1}\\
\tilde{w}_x(0,t)&=&\mathbf{0}\;\;\;\tilde{w}(1,t)=\begin{pmatrix}
-L\tilde{u}_x(1,t)\\
0
\end{pmatrix}\label{hehebc1}
\end{eqnarray}
Using the following (invertible) transformation
\begin{eqnarray}
\tilde{w}(x,t)=\tilde{\gamma}(x,t)-\int_x^1\! P(x,y)\tilde{\gamma}(y,t)\,\mathrm{d}y\label{coltrans1}
\end{eqnarray}
we transform the error system into (\ref{targ1})-(\ref{bctarg1}). By plugging the transformation (\ref{coltrans1}) into the error system (\ref{hehe1})-(\ref{hehebc1}), using the target system (\ref{targ1})-(\ref{bctarg1}), and integration by parts, the kernels must satisfy the following system
\begin{eqnarray}
\bar{P}_{xx}(x,y)-\bar{P}_{yy}(x,y)&=&-\Psi\bar{P}(x,y)\label{colobsker1}\\
\bar{P}(x,x)&=&\left(0\;-\frac{\lambda_1}{2}x\;-\frac{\lambda_2}{2}x\;0\right)^T\\
\bar{P}_x(0,y)&=&\mathbf{0}\label{colobsbc4}
\end{eqnarray}
The gains are given by
\begin{eqnarray}
p_1(x)&=&P^{uu}(x,1),\;\;p_2(x)=P^{vu}(x,1),\;\;L=0\label{colgain}
\end{eqnarray}
\begin{theorem}
Let $P(x,y)$ be the solution to the system (\ref{colobsker1})-(\ref{colobsbc4}). Then for any initial condition $\tilde{w}_0\in L^2(0,1)$ which compatible with the boundary condition (\ref{colerrorbc4}), the system (\ref{colerror1}) with $p_1(x)$, $p_2(x)$, and $L$ are given by (\ref{colgain}), has a unique classical solution $\tilde{w}\in C^{2,1}((0,1)\times(0,\infty))$. Additionally, the origin, $\tilde{w}\equiv0$, is exponentially stable in the $L^2$-sense.
\end{theorem}

\section{OUTPUT-FEEDBACK STABILIZATION}

The exponentially convergent observer developed in the last section is independent of the control input and can be used with any controller. In this section we combine it with the backstepping controller developed in section 2 to solve the output-feedback problems.

\subsection{Anti-collocated setup}

For anti-collocated setup, we use the following transformation
\begin{eqnarray}
\hat{\gamma}(x,t)=\hat{w}(x,t)-\int_0^x\! K(x,y)\hat{w}(y,t)\,\mathrm{d}y\label{outtrans1}
\end{eqnarray}

\begin{lemma}
The transformations (\ref{outtrans1}) and (\ref{acoltrans1}) map the observer (\ref{obsmain1})-(\ref{obsbc4}) and (\ref{anerror1})-(\ref{anerrorbc4}) into the following system
\begin{eqnarray}
\hat{\gamma}_t(x,t)&=&\hat{\gamma}_{xx}(x,t)+\left(p(x)-Q(x)\right)\tilde{\alpha}(0,t)\\
\hat{\gamma}_x(0,t)&=&\mathbf{0}\;\;\;\hat{\gamma}(1,t)=\mathbf{0}\\
\tilde{\gamma}_t(x,t)&=&\tilde{\gamma}_{xx}(x,t)\\
\tilde{\gamma}_x(0,t)&=&\mathbf{0}\;\;\;\tilde{\gamma}(1,t)=\mathbf{0}
\end{eqnarray}
where
\begin{eqnarray}
Q(x)=\begin{pmatrix}
Q_1(x)\\
Q_2(x)
\end{pmatrix}=\int_0^x\! K(x,y)p(y)\,\mathrm{d}y
\end{eqnarray}
\end{lemma}
\begin{lemma}
For any $(\hat{w}_0,\tilde{w}_0)\in L^2(0,1)$, the system $(\hat{w},\tilde{w})$ is exponentially stable in the $L^2$-norm.
\end{lemma}
\begin{proof}
Consider the following Lyapunov function
\begin{eqnarray}
V(t)&=&\frac{A}{2}\int_0^1\! \tilde{\alpha}(x,t)^2\,\mathrm{d}x+\frac{1}{2}\int_0^1\! \hat{\alpha}(x,t)^2\,\mathrm{d}x\nonumber\\
&&+\frac{1}{2}\int_0^1\! \tilde{\beta}(x,t)^2\,\mathrm{d}x+\frac{1}{2}\int_0^1\! \hat{\beta}(x)^2\,\mathrm{d}x\label{lya}
\end{eqnarray}
where $A$ is a positive constant to be determined latter. Taking the time derivative of (\ref{lya}), we have
\begin{eqnarray}
\dot{V}(t) &=& -A\int_0^1\! \tilde{\alpha}_x(x,t)^2\,\mathrm{d}x-\int_0^1\! \hat{\alpha}_x(x,t)^2\,\mathrm{d}x\nonumber\\
&&-\int_0^1\! \tilde{\beta}_x(x,t)^2\,\mathrm{d}x-\int_0^1\! \hat{\beta}_x(x,t)^2\,\mathrm{d}x\nonumber\\
&&+\int_0^1\! \hat{\alpha}(x,t)\left(p_1(x)-Q_1(x)\right)\tilde{\alpha}(0,t)\,\mathrm{d}x\nonumber\\
&&+\int_0^1\! \hat{\beta}(x,t)\left(p_2(x)-Q_2(x)\right)\tilde{\alpha}(0,t)\,\mathrm{d}x
\end{eqnarray}
Next, we estimate the last two terms
\begin{eqnarray}
&&\tilde{\alpha}(0,t)\int_0^1\! \hat{\alpha}(x,t)\left(p_1(x)-Q_1(x)\right)\,\mathrm{d}x\nonumber\\
&\leq&\frac{1}{4}\int_0^1\! \hat{\alpha}_x(x,t)^2\,\mathrm{d}x+C^2\int_0^1\! \tilde{\alpha}_x(x,t)^2\,\mathrm{d}x
\end{eqnarray}
\begin{eqnarray}
&&\tilde{\alpha}(0,t)\int_0^1\! \hat{\beta}(x,t)\left(p_2(x)-Q_2(x)\right)\,\mathrm{d}x\nonumber\\
&\leq&\frac{1}{4}\int_0^1\! \hat{\beta}_x(x,t)^2\,\mathrm{d}x+D^2\int_0^1\! \tilde{\alpha}_x(x,t)^2\,\mathrm{d}x
\end{eqnarray}
where
\begin{eqnarray}
C=\max_{x\in[0,1]}\left\{p_1(x)-Q_1(x)\right\}\\
D=\max_{x\in[0,1]}\left\{p_2(x)-Q_2(x)\right\}
\end{eqnarray}
then we have
\begin{eqnarray}
\dot{V} &\leq& -\left(A-C^2-D^2\right)\int_0^1\! \tilde{\alpha}_x(x,t)^2\,\mathrm{d}x\nonumber\\
&&-\frac{3}{4}\int_0^1\! \hat{\alpha}_x(x,t)^2\,\mathrm{d}x-\int_0^1\! \tilde{\beta}_x(x,t)^2\,\mathrm{d}x\nonumber\\
&&-\frac{3}{4}\int_0^1\! \hat{\beta}_x(x,t)^2\,\mathrm{d}x\nonumber\\
&\leq& -\frac{1}{4}\left(A-C^2-D^2\right)\int_0^1\! \tilde{\alpha}_x(x,t)^2\,\mathrm{d}x\nonumber\\
&&-\frac{1}{8}\int_0^1\! \hat{\alpha}_x(x,t)^2\,\mathrm{d}x-\frac{1}{8}\int_0^1\! \tilde{\beta}_x(x,t)^2\,\mathrm{d}x\nonumber\\
&&-\frac{1}{8}\int_0^1\! \hat{\beta}_x(x,t)^2\,\mathrm{d}x
\end{eqnarray}
Taking $A=2\left(C^2+D^2\right)$, we get $\dot{V}\leq-\frac{1}{4}V$. This shows the system $(\hat{\gamma},\tilde{\gamma})$ is exponentially stable. Since the transformation (\ref{outtrans1}) is invertible, this concludes the proof.
\end{proof}
The summary of the output-feedback stabilization for anti-collocated setup can be stated in the following theorem.
\begin{theorem}
Let $K(1,y)$ can be obtained from (\ref{ker1})-(\ref{bcker6}), and $p_1(x)$, $p_2(x)$ and $L$ are given by (\ref{angain}). Then for any $(w_0,\hat{w}_0)\in L^2(0,1)$ the system consisting of plant (\ref{main1})-(\ref{bc2}), the controller
\begin{eqnarray}
U(t)=\int_0^1\! K(1,y)\hat{w}(y,t)\,\mathrm{d}y
\end{eqnarray}
and the observer
\begin{eqnarray}
\hat{w}_t(x,t)&=&\hat{w}_{xx}(x,t)+\Sigma\hat{w}(x,t)+p(x)\tilde{u}(0,t)\\
\hat{w}_x(0,t)&=&\mathbf{0}\;\;\;\hat{w}(1,t)=\int_0^1\! K(1,y)\hat{w}(y,t)\,\mathrm{d}y
\end{eqnarray}
has a unique classical solution $\left(w,\hat{w}\right)\in C^{2,1}((0,1)\times(0,\infty))$ and is exponentially stable at the origins, i.e., $\left(w,\hat{w}\right)\equiv0$, in the $L^2$-sense.
\end{theorem}

\subsection{Collocated setup}

For collocated setup, applying the same steps as in the previous subsection, we have the following result.

\begin{theorem}
Let $K(1,y)$ can be obtained from (\ref{ker1})-(\ref{bcker6}), and $p_1(x)$, $p_2(x)$, and $L$ are given by (\ref{colgain}). Then for any $(w_0,\hat{w}_0)\in L^2(0,1)$ the system consisting of plant (\ref{main1})-(\ref{bc2}), the controller
\begin{eqnarray}
U(t)=\int_0^1\! K(1,y)\hat{w}(y,t)\,\mathrm{d}y
\end{eqnarray}
and the observer
\begin{eqnarray}
\hat{w}_t(x,t)&=&\hat{w}_{xx}(x,t)+\Sigma\hat{w}(x,t)+p(x)\tilde{u}_x(1,t)\\
\hat{w}_x(0,t)&=&\mathbf{0}\;\;\;\hat{w}(1,t)=\int_0^1\! K(1,y)\hat{w}(y,t)\,\mathrm{d}y
\end{eqnarray}
has a unique classical solution $\left(w,\hat{w}\right)\in C^{2,1}((0,1)\times(0,\infty))$ and is exponentially stable at the origins, i.e., $\left(w,\hat{w}\right)\equiv0$, in the $L^2$-sense.
\end{theorem}

\section{WELL POSEDNESS OF THE KERNEL EQUATION}

Remark that all kernel equations almost have the same structures. They are consisted of two separate hyperbolic equations. Therefore, without loss of generality, let us consider the first two kernel equations in (\ref{ker1})-(\ref{bcker6})
\begin{eqnarray}
K_{xx}^{uu}(x,y)&-&K_{yy}^{uu}(x,y)=\lambda_2K^{uv}(x,y)\label{din}\\
K_{xx}^{uv}(x,y)&-&K_{yy}^{uv}(x,y)=\lambda_1K^{uu}(x,y)\\
K^{uu}(x,x)&=&0\;\;\;\;\;\;\;\;\;\;K_y^{uu}(x,0)=0\\
K^{uv}(x,x)&=&-\frac{\lambda_1}{2}x\;\;\;K_y^{uv}(x,0)=0\label{lov}
\end{eqnarray}
The idea is to change the PDEs into integral equations and solving them using a successive approximation method.
\begin{lemma}
Let $\xi = x+y$, $\eta=x-y$, and denote $G(\xi,\eta)=K^{uu}(x,y)$ and $H(\xi,\eta)=K^{uv}(x,y)$. Then (\ref{din})-(\ref{lov}) can be transformed into the following integral equations
\begin{eqnarray}
G(\xi,\eta)&=&\frac{\lambda_2}{2}\int_0^{\eta}\! \int_0^{\tau}\! H(\tau,s)\,\mathrm{d}s\,\mathrm{d}\tau\nonumber\\
&&+\frac{\lambda_2}{4}\int_{\eta}^{\xi}\! \int_0^{\eta}\! H(\tau,s)\,\mathrm{d}s\,\mathrm{d}\tau\label{Gint}\\
H(\xi,\eta)&=&-\frac{\lambda_1}{4}(\xi+\eta)+\frac{\lambda_1}{2}\int_0^{\eta}\! \int_0^{\tau}\! G(\tau,s)\,\mathrm{d}s\,\mathrm{d}\tau\nonumber\\
&&+\frac{\lambda_1}{4}\int_{\eta}^{\xi}\! \int_0^{\eta}\! G(\tau,s)\,\mathrm{d}s\,\mathrm{d}\tau\label{Hint}
\end{eqnarray}
for $(\xi,\eta)\in\Xi=\{(\xi,\eta)\in\mathbb{R}^2|0\leq \xi\leq 2,\; 0\leq\eta\leq\min(\xi,2-\xi)\}$.
\end{lemma}

\begin{proof}
By direct calculation we transform (\ref{din})-(\ref{lov}) into
\begin{eqnarray}
G_{\xi\eta}(\xi,\eta)&=&\frac{\lambda_2}{4}H(\xi,\eta)\label{G}\\
H_{\xi\eta}(\xi,\eta)&=&\frac{\lambda_1}{4}G(\xi,\eta)\label{H}\\
G(\xi,0)&=&0\;\;\;\;\;\;\;\;\;\;G_{\xi}(\xi,\xi)=G_{\eta}(\xi,\xi)\label{Gfirst}\\
H(\xi,0)&=&-\frac{\lambda_1}{4}\xi\;\;\;H_{\xi}(\xi,\xi)=H_{\eta}(\xi,\xi)\label{Hfirst}
\end{eqnarray}
Integrating (\ref{G}) and (\ref{H}) with respect to $\eta$ from $0$ to $\eta$, using the first boundary conditions of (\ref{Gfirst}) and (\ref{Hfirst}), yield
\begin{eqnarray}
G_{\xi}(\xi,\eta)&=&\frac{\lambda_2}{4}\int_0^{\eta}\! H(\xi,s)\,\mathrm{d}s\\
H_{\xi}(\xi,\eta)&=&-\frac{\lambda_1}{4}+\frac{\lambda_1}{4}\int_0^{\eta}\! G(\xi,s)\,\mathrm{d}s
\end{eqnarray}
Integrating the above equations with respect to $\xi$ from $\eta$ to $\xi$, we have
\begin{eqnarray}
G(\xi,\eta)&=&G(\eta,\eta)+\frac{\lambda_2}{4}\int_{\eta}^{\xi}\! \int_0^{\eta}\! H(\tau,s)\,\mathrm{d}s\,\mathrm{d}\tau\label{firstint}\\
H(\xi,\eta)&=&H(\eta,\eta)-\frac{\lambda_1}{4}(\xi-\eta)\nonumber\\
&&+\frac{\lambda_1}{4}\int_{\eta}^{\xi}\! \int_0^{\eta}\! G(\tau,s)\,\mathrm{d}s\,\mathrm{d}\tau\label{secondint}
\end{eqnarray}
To find $G(\eta,\eta)$ and $H(\eta,\eta)$, we use the second boundary conditions of (\ref{Gfirst}) and (\ref{Hfirst})
\begin{eqnarray}
\frac{\mathrm{d}}{\mathrm{d}\xi}G(\xi,\xi) &=& 2G_{\xi}(\xi,\xi)=\frac{\lambda_2}{2}\int_0^{\xi}\! H(\xi,s)\,\mathrm{d}s\\
\frac{\mathrm{d}}{\mathrm{d}\xi}H(\xi,\xi) &=& 2H_{\xi}(\xi,\xi)=-\frac{\lambda_1}{2}+\frac{\lambda_1}{2}\int_0^{\xi}\! G(\xi,s)\,\mathrm{d}s
\end{eqnarray}
Integrating the above equations for $\xi=\eta$, we have
\begin{eqnarray}
G(\eta,\eta) &=&\frac{\lambda_2}{2}\int_0^{\eta}\! \int_0^{\tau}\! H(\tau,s)\,\mathrm{d}s\,\mathrm{d}\tau\\
H(\eta,\eta) &=&-\frac{\lambda_1}{2}\eta+\frac{\lambda_1}{2}\int_0^{\eta}\! \int_0^{\tau}\! G(\tau,s)\,\mathrm{d}s\,\mathrm{d}\tau
\end{eqnarray}
Plugging this into (\ref{firstint})-(\ref{secondint}) yields (\ref{Gint})-(\ref{Hint}).
\end{proof}
Equation (\ref{Gint}) and (\ref{Hint}) are solved using the method of successive approximations.
\begin{theorem}
System (\ref{din})-(\ref{lov}) has a unique $C^2(\Upsilon)$ solution and the bound on the solutions are given by
\begin{eqnarray}
|K^{uu}(x,y)|\leq Me^{2Mx}\;\;\;\;\;|K^{uv}(x,y)|\leq Me^{2Mx}
\end{eqnarray}
where $M$ is a sufficiently large constant.
\end{theorem}

\begin{proof}
Let $
J=\begin{pmatrix}
G\\
H
\end{pmatrix}
$
and define the following sequence
\begin{eqnarray}
J^0(\xi,\eta)&=&0\\
J^{n+1}(\xi,\eta)&=&\Theta(\xi,\eta)+\Omega[J^{n}](\xi,\eta)\label{keng}
\end{eqnarray}
where $\Theta(\xi,\eta)=\begin{pmatrix}
0\\
-\frac{1}{4}(\xi+\eta)
\end{pmatrix}$. For $n\geq0$ define the increment $\Delta J^{n+1}(\xi,\eta)=J^{n+1}(\xi,\eta)-J^n(\xi,\eta)$, with $\Delta J^1(\xi,\eta)=\Theta(\xi,\eta)$. Since $\Omega$ is linear, then
\begin{eqnarray}
\Delta J^{n+1}(\xi,\eta)=\Omega[\Delta J^n](\xi,\eta)\label{ais}
\end{eqnarray}
If the limit exists, then $J(\xi,\eta)=\lim_{n\rightarrow\infty}J^n(\xi,\eta)$ is the solution to the integral equations, or using the definition of $\Delta J^{n+1}(\xi,\eta)$, it follows that
\begin{eqnarray}
J(\xi,\eta) = \sum_{n=0}^{\infty}\Delta J^{n+1}(\xi,\eta)\label{seri}
\end{eqnarray}
Suppose
\begin{eqnarray}
|\Delta J^n(\xi,\eta)|\leq M^{n+1}\frac{(\xi+\eta)^n}{n!}\label{induc}
\end{eqnarray}
Then from (\ref{ais})and (\ref{Gint}), we have
\begin{eqnarray}
|\Delta G^{n+1}(\xi,\eta)|&\leq&\frac{M^{n+1}}{n!}\left\{\frac{\lambda_2}{2}\int_0^{\eta}\! \int_0^{\tau}\! (\tau+s)^n\,\mathrm{d}s\,\mathrm{d}\tau\right.\nonumber\\
&&\left.+\frac{\lambda_2}{4}\int_{\eta}^{\xi}\! \int_0^{\eta}\! (\tau+s)^n\,\mathrm{d}s\,\mathrm{d}\tau\right\}\nonumber\\
&\leq&\frac{M^{n+1}}{n!}\frac{3}{4}\lambda_2\frac{(\xi+\eta)^{n+1}}{n+1}\nonumber\\
&\leq&M^{n+2}\frac{(\xi+\eta)^{n+1}}{(n+1)!}\label{mia}
\end{eqnarray}
Similarly, from (\ref{ais}) and (\ref{Hint}), we have
\begin{eqnarray}
|\Delta H^{n+1}(\xi,\eta)|&\leq&M^{n+2}\frac{(\xi+\eta)^{n+1}}{(n+1)!}
\end{eqnarray}
Hence
\begin{eqnarray}
|\Delta J^{n+1}(\xi,\eta)|\leq M^{n+2}\frac{(\xi+\eta)^{n+1}}{(n+1)!}
\end{eqnarray}
Therefore, by induction (\ref{induc}) holds and the series (\ref{seri}) converges uniformly in $\Xi$ and bounded by $|J(\xi,\eta)|\leq Me^{M(\xi+\eta)}$. Furthermore, remark that from (\ref{Gint})-(\ref{Hint}), $\Delta J^{n+1}\in C^2(\Xi)$. Hence, the solution $J(\xi,\eta)$ is a twice continuously differentiable function. For uniqueness, let $J_1(\xi,\eta)$ and $J_2(\xi,\eta)$ are different solution of (\ref{keng}). Define $\tilde{J}(\xi,\eta)=J_1(\xi,\eta)-J_2(\xi,\eta)$, then $|\tilde{J}(\xi,\eta)|\leq2Me^{2M}$. By following the same estimates as in (\ref{mia}), yields
\begin{eqnarray}
|\tilde{J}(\xi,\eta)|\leq2M^{n+1}e^{2M}\frac{(\xi,\eta)^n}{n!}\rightarrow_{n\rightarrow\infty}0
\end{eqnarray}
Thus, $\tilde{J}\equiv0$, which means that (\ref{seri}) is a unique solution.
\end{proof}

Here, we give an example for $\lambda_1=20$ and $\lambda_2=10$. \textbf{Fig.~\ref{kuu}} and \textbf{Fig.~\ref{kuv}} are the solution of the kernel equations (\ref{din})-(\ref{lov}), respectively. The figures show the kernels are smooth functions. The gains for the full-state feedback law (\ref{con1}) is given in \textbf{Fig.~\ref{gain}}.

\begin{figure}[h!]
  \centering
      \includegraphics[width=0.5\textwidth]{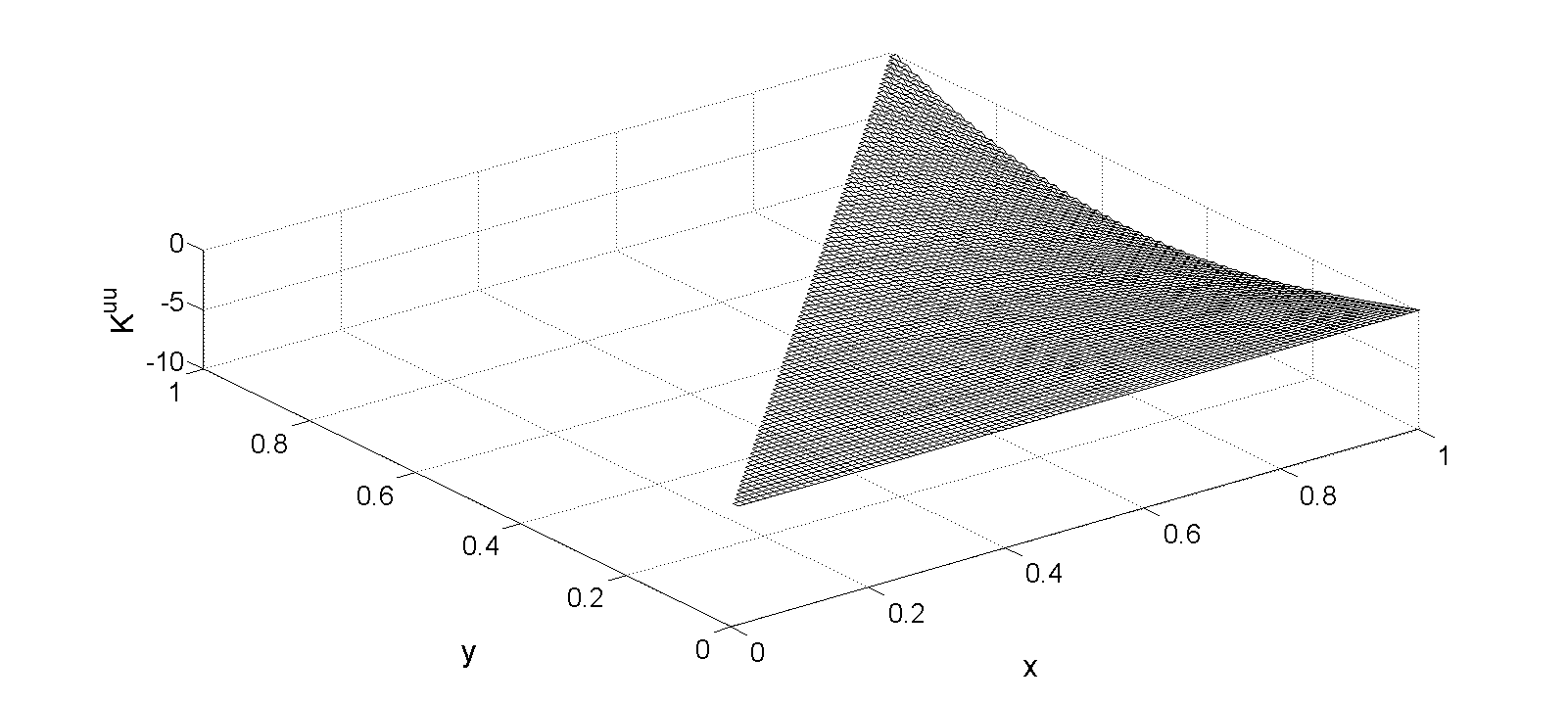}
  \caption{Kernel $K^{uu}(x,y)$.}
\label{kuu}
\end{figure}

\begin{figure}[h!]
  \centering
      \includegraphics[width=0.5\textwidth]{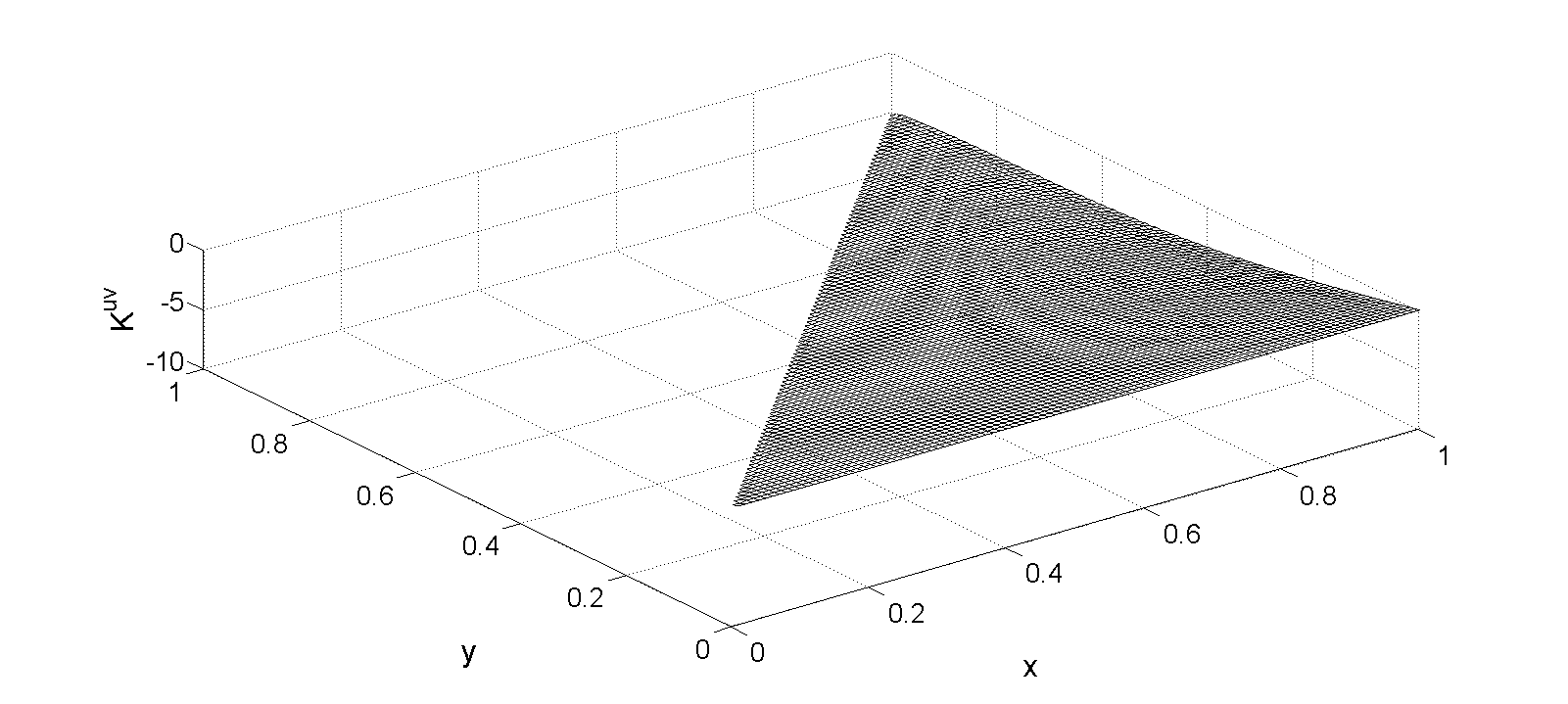}
  \caption{Kernel $K^{uv}(x,y)$.}
\label{kuv}
\end{figure}

\begin{figure}[h!]
  \centering
      \includegraphics[width=0.5\textwidth]{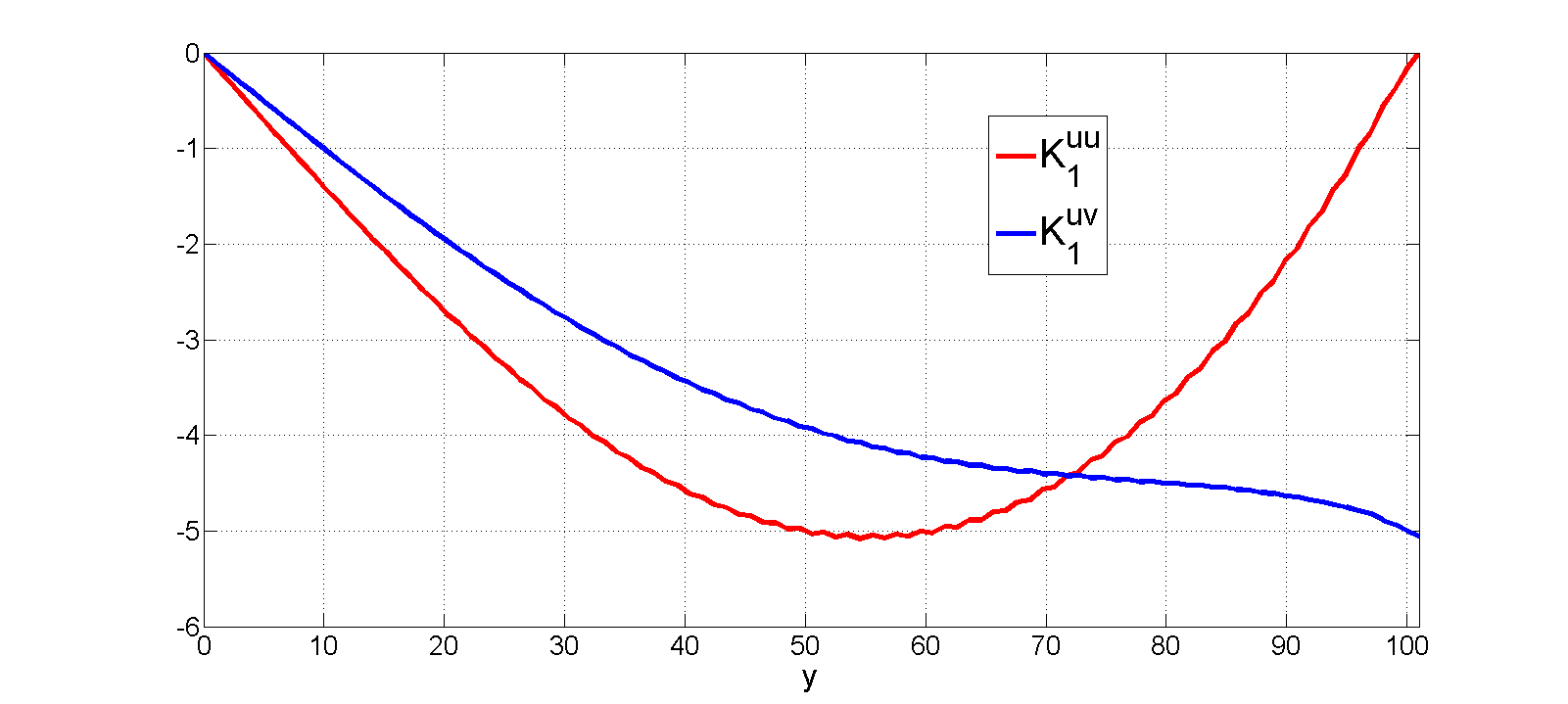}
  \caption{Gain $K^{uu}(1,y)$ and $K^{uv}(1,y)$.}
\label{gain}
\end{figure}

\section{CONCLUSIONS}

We have solved boundary output-feedback stabilization problem for a class of two-component linear parabolic systems for both anti-collocated and collocated setup where Volterra integral transformations of the second kind are employed. The main challenge is to show existence and invertibility of the transformations for both control and observer designs. The well-posedness of the kernel equations is proved by transforming the PDEs into integral equations and using the method of successive approximations. In this paper, we assume the components to have the same diffusion coefficients; hence, a natural step will be to study parabolic systems with different coefficients.

\end{document}